\documentclass[11pt]{article}
\usepackage{amsmath,amssymb,amscd,amsthm,graphicx,enumerate,enumitem,framed}
\usepackage[a4paper]{geometry}
\usepackage{hyperref}

\title{Singularities of mean convex level set flow in general ambient manifolds}
\author{Robert Haslhofer and Or Hershkovits}
\date{\today}

\usepackage{amsfonts}

\numberwithin{equation}{section}
  \theoremstyle{plain}
 \newtheorem{theorem}[equation]{Theorem}
\newtheorem{proposition}[equation]{Proposition}
\newtheorem{corollary}[equation]{Corollary}
 \newtheorem{lemma}[equation]{Lemma}
 \newtheorem{claim}[equation]{Claim}
 \theoremstyle{remark}
 \newtheorem{remark}[equation]{Remark}
 \theoremstyle{remark}
\theoremstyle{definition}

\newcommand*{\rom}[1]{\expandafter\@slowromancap\romannumeral #1@}

\newcommand{\de}{\varepsilon}
\newcommand{\p}{\phi}

\newcommand{\D}{\nabla}
\newcommand{\Lap}{\Delta}
\newcommand{\eps}{\varepsilon}
\newcommand{\s}{\sigma}
\newcommand{\ka}{\kappa}

\newcommand{\Rc}{\textrm{Rc}}
\newcommand{\graph}{\mathrm{graph}}

\providecommand{\abs}[1]{\lvert #1\rvert}

\begin{document}
\maketitle

\begin{abstract}
We prove two new estimates for the level set flow of mean convex domains in Riemannian manifolds. Our estimates give control - exponential in time - for the infimum of the mean curvature, and the ratio between the norm of the second fundamental form and the mean curvature. 
In particular, the estimates remove a stumbling block that has been left after the work of White \cite{white_size,white_nature,white_subsequent}, and Haslhofer-Kleiner \cite{HK_meanconvex}, and thus allow us to extend the structure theory for mean convex level set flow to general ambient manifolds of arbitrary dimension.
\end{abstract}

\section{Introduction}
Let $N$ be a Riemannian manifold. For any mean convex domain $K_0 \subset N$ we consider the level set flow $\{K_t\}_{t\geq 0}$ starting at $K_0$, i.e. the maximal family of closed sets starting at $K_0$ that satisfies the avoidance principle when compared with any smooth mean curvature flow \cite{evans-spruck,CGG,Ilmanen}. The level set flow of $K_0$ coincides with the smooth mean curvature flow of $K_0$ for as long as the latter is defined, but provides a canonical way to continue the evolution beyond the first singular time. Mean convexity is preserved also beyond the first singular time in the  sense that $K_{t_2}\subseteq K_{t_1}$ whenever $t_2\geq t_1$.

In the last 15 years, Brian White developed a deep regularity and structure theory for mean convex level set flow \cite{white_size,white_nature,white_subsequent}, and recently the first author and Kleiner gave a new treatment of this theory \cite{HK_meanconvex}. Concerning the size of the singular set, White proved that the singular set $\mathcal{S}\subset N^n\times \mathbb{R}$ of any mean convex flow has parabolic Hausdorff dimension at most $n-2$ \cite[Thm. 1.1]{white_size}, see also \cite[Thm. 1.15]{HK_meanconvex}. Concerning the structure of the singular set, the main assertion one wants to prove is that all blowup limits of a mean convex flow are smooth and convex until they become extinct. In particular, one wants to conclude that all tangent flows of a mean convex flow are round shrinking spheres, round shrinking cylinders, or static planes of multiplicity one.
While the theorem about the size of the singular set is known in full generality, the structure theorem has been proved up to now only under some additional assumptions \cite[Thm. 1]{white_nature}, \cite[Thm. 3]{white_subsequent} and \cite[Thm. 1.14]{HK_meanconvex}. Namely one has to restrict either to blowups at the first singular time, or to low dimensions, or to the case where the ambient manifold is Euclidean space.

As explained in \cite[Appendix B]{white_subsequent}, the missing step to extend the structure theorem to general ambient manifolds of arbitrary dimension is to prove that the ratio between the smallest principal curvature $\lambda_1$ and the mean curvature $H$ has a finite lower bound on the regular points contained in any compact subset of space-time.

The purpose of this work is to remove this stumbling block. To this end, we prove two new estimates for the level set flow of mean convex domains in Riemannian manifolds.

To state our estimates, we denote by $\partial K_t^\textrm{reg}$ the set of regular boundary points at time $t$.
Our first main estimate gives a lower bound for the mean curvature.

\begin{theorem}[Lower bound for $H$]\label{estimate1}
There exist constants $H_0=H_0(K_0)>0$ and $\rho=\rho(K_0)<\infty$ such that
\begin{equation}
\inf_{\partial K_t^\textrm{reg}}H\geq H_0 e^{-\rho t}.
\end{equation}
\end{theorem}
Our estimate from Theorem \ref{estimate1}, as well as our second main estimate below, depends exponentially on time. It is clear from simple examples (e.g. flows in hyperbolic space), that this exponential behavior in time is the best one can possibly get.

Our second main estimate controls the ratio between the norm of the second fundamental form and the mean curvature.

\begin{theorem}[Upper bound for $\abs{A}/H$]\label{estimate2}
There exist constants $C=C(K_0)<\infty$ and $\rho=\rho(K_0)<\infty$ such that
\begin{equation}
\inf_{ \partial K_t^\textrm{reg}}\frac{\abs{A}}{H}\leq C e^{\rho t}.
\end{equation}
\end{theorem}
Theorem \ref{estimate2} shows that all principal curvatures are controlled by the mean curvature, and thus in particular provides a (two-sided) bound for the ratio $\lambda_1/H$.
As explained above, this exactly fills in the missing piece that is needed to extend the structure theorem for mean convex level set flow to the general case without restrictions on subsequent singularities, the ambient manifold, and the dimension. We thus obtain:

\begin{theorem}[Structure theorem]\label{structurethm}
Let $K_0\subset N$ be a mean convex domain in a Riemannian manifold. Then all blowup limits of its level set flow $\{K_t\}_{t\geq 0}$ are smooth and convex until they become extinct. In particular, all backwardly selfsimilar blowup limits are round shrinking spheres, round shrinking cylinders, or static planes of multiplicity one.
\end{theorem}

Theorem \ref{structurethm} gives a general description of the nature of singularities of mean convex level set flow in arbitrary ambient manifolds. As mentioned above, this generalizes the structure theorems from \cite[Thm. 1]{white_nature}, \cite[Thm. 3]{white_subsequent} and \cite[Thm. 1.14]{HK_meanconvex}.\\

\noindent\textbf{Applications.} Let us now discuss some applications of the above theorems.

Our first application concerns topological changes in mean convex mean curvature flow. In \cite{White_topo_change}, White proved that under mean convex level set flow elements of the $m$-th homotopy group of the complementary region can die only if there is a shrinking $S^k\times \mathbb{R}^{n-1-k}$ singularity for some $k\leq m$, assuming that $n\leq7$ or that the ambient manifold is Euclidean. Thanks to Theorem \ref{structurethm} we can remove the assumption on the dimension and the ambient manifold, and thus obtain:

\begin{corollary}[Topological change]
Let $K_0\subset N^{n}$ be a mean convex domain in a Riemannian manifold. If for some $0\leq t_1<t_2$ there is a map of the $m$-sphere into $N\setminus K_{t_1}$ that is homotopically trivial in $N\setminus K_{t_2}$ but not in $N\setminus K_{t_1}$, then at some $t\in [t_1,t_2)$ there is a singularity of the flow at which the tangent flow is a shrinking $S^k\times \mathbb{R}^{n-1-k}$ for some $k\in \{1,\ldots,m\}$.
\end{corollary}

Our second application concerns the estimates for mean convex level set flow in the setting of Haslhofer-Kleiner \cite{HK_meanconvex}. These estimates are based on the noncollapsing condition that each boundary point admits interior and exterior balls of radius comparable to the reciprocal of the mean curvature at that point \cite{white_size,sheng_wang,andrews1}. It has been unknown up to now if this noncollapsing condition holds for mean convex level set flow in general ambient manifolds of arbitrary dimension. Combining Theorem \ref{estimate1} and Theorem \ref{structurethm} we can answer this in the affirmative:

\begin{corollary}[Noncollapsing]\label{cor_noncollapsing}
Let $K_0\subset N^{n}$ be a mean convex domain in a Riemannian manifold. Then there exists a positive nonincreasing function $\alpha:[0,\infty)\to (0,\infty)$ such that each $p\in \partial K_t^\textrm{reg}$ admits an interior and exterior ball tangent at $p$ of radius at least $\alpha(t)/H(p,t)$. In particular, all estimates from \cite{HK_meanconvex} apply in the setting of mean convex level set flow in general ambient manifolds of arbitrary dimension.
\end{corollary}

\begin{remark}
We conjecture that the conclusion of Corollary \ref{cor_noncollapsing} actually holds for some $\alpha(t)\geq \alpha_0e^{-\rho t}$ for some $\alpha_0=\alpha_0(K_0)>0$ and $\rho=\rho(K_0)<\infty$. It would also be interesting to find a proof of the noncollapsing which is independent of Theorem \ref{structurethm}.
\end{remark}

Our third application concerns a sharp estimate for the inscribed and outer radius for mean convex level set flow in Riemannian manifolds. In \cite{Brendle_inscribed} and \cite{Brendle_inscribed_manifolds}, Brendle proved sharp bounds for the inscribed radius and outer radius at points in a smooth mean convex mean curvature flow where the mean curvature is large. The first author and Kleiner \cite{HK_inscribed} found a shorter proof of Brendle's estimate, which also works in the nonsmooth setting provided that one has some noncollapsing parameter to get started. Thanks to Corollary \ref{cor_noncollapsing} the argument from \cite{HK_inscribed} is applicable for mean convex level set flow in general ambient manifolds, and we thus obtain:

\begin{corollary}[Sharp estimate for inscribed and outer radius]
Let $K_0\subset N$ be a mean convex domain in a Riemannian manifold. Then for any positive nonincreasing function $\delta:[0,\infty)\to (0,\infty)$, there exists a positive nonincreasing function $H_0:[0,\infty)\to (0,\infty)$ depending only on $K_0$ and $\delta$ such that every $p\in K_t^\textrm{reg}$ with $H(p,t)\geq H_0(t)$ admits an interior ball of radius at least $\tfrac{1}{(1+\delta(t))H(p,t)}$ and an exterior ball of radius at least $\tfrac{1}{\delta(t)H(p,t)}$.
\end{corollary}

\noindent\textbf{Outline.} To finish this introduction, let us now describe some of the key ideas behind the proofs of our two main estimates (Theorem \ref{estimate1} and Theorem \ref{estimate2}).

The estimates are very easy to prove for smooth flows, so let us start by explaining this:
First, from the evolution equation for the mean curvature \cite[Cor. 3.5]{Huisken_manifolds},
\begin{equation}
\partial_t H=\Lap H+\abs{A}^2 H + \textrm{Rc}(\nu,\nu)H,
\end{equation}
and the maximum principle, one sees that the minimum of the mean curvature can deteriorate at most exponentially in time. Second, combining the evolution equation for the square norm of the second fundamental form \cite[Cor. 3.5]{Huisken_manifolds},
\begin{align}
\partial_t \abs{A}^2=&\Lap\abs{A}^2-2\abs{\nabla A}^2+2\abs{A}^4+2 \textrm{Rc}(\nu,\nu)\abs{A}^2\nonumber\\
&-4(h_{ij}h_{jm}R_{mlil}-h_{ij}h_{lm}R_{milj})-2h_{ij}(\nabla R_{0lil}+\nabla_l R_{0ijl}),
\end{align}
and the evolution equation for the mean curvature, one sees that the maximum of $\abs{A}/H$ increases at most exponentially in time.
We emphasize that the above estimates crucially rely on one another. Namely, to control the reaction terms in the evolution for $\abs{A}/H$ we need the lower bound for $H$ from the first step.

Having sketched the argument in the smooth case, the main difficulty is to generalize this argument to the level set flow beyond the first singular time. As in White \cite{white_subsequent}, a natural first approach to try would be to use elliptic regularization. Recall that the time of arrival function $u$ of a mean convex flow $\{ K_t\}_{t\geq 0}$ is defined by $u(x)=t$ if and only if $x\in\partial K_t$. For mean convex flows in Euclidean space, the time of arrival function $u:K_0\to \mathbb{R}$ is a bounded real valued function with domain $K_0$, and can be approximated by solutions of the Dirichlet problem
\begin{align}
\textrm{div}\left(\frac{Du_\eps}{\sqrt{\eps^2+\abs{Du_\eps}^2}}\right)+\frac{1}{\sqrt{\eps^2+\abs{Du_\eps}^2}}=0  & \qquad \textrm{in} \,\,\textrm{Int}(K_0),\nonumber\\
u_\eps=0 &\qquad \textrm{on} \,\,\partial K_0.
\end{align}
The elliptic regularization technique has been known for a long time \cite{evans-spruck,CGG}, see also \cite{Ilmanen}, and arguing as in \cite{white_subsequent,HK_meanconvex} can be used to prove that the two main estimates (with $\rho=0$) hold for the level set flow in Euclidean space. However, extending these arguments to level set flow in Riemannian manifolds is not straightforward.

The key difference between level set flow in Euclidean space and level set in general ambient manifolds, is that in the latter case the flow generally does not become extinct in finite time, but converges to a nonempty limit $K_\infty$ for $t\to\infty$. Consequently, the time of arrival function $u$ is only defined on the set $K_0\setminus K_\infty$. Thus, it is (a) not clear a priori how to approximate $u$ by smooth solutions, and (b) even if one succeeds in approximating $u$ by smooth solutions it is not obvious how to prove our main estimates using the approximators, since one would have to somehow bring in the exponential in time factor and would have to cut off all quantities under consideration for $t\to \infty$. 

To overcome the above difficulties, we consider a new double-approximation scheme. Namely, we consider functions $u_{\eps,\sigma}$ solving the 
Dirichlet problem
\begin{align}\label{double_reg}
\textrm{div}\left(\frac{Du_{\eps,\sigma}}{\sqrt{\eps^2+\abs{Du_{\eps,\sigma}}^2}}\right)+\frac{1}{\sqrt{\eps^2+\abs{Du_{\eps,\sigma}}^2}}&=\sigma u_{\eps,\sigma}  &  \textrm{in} \,\,\textrm{Int}(K_0),\nonumber\\
u_{\eps,\sigma}&= 0 & \textrm{on} \,\,\partial K_0.
\end{align}
The idea, inspired in part by the Schoen-Yau proof of the positive mass theorem \cite{SchoenYau}, is that for $\sigma>0$ the maximum principle gives the a-priori sup bound $u_{\eps,\sigma}\leq \frac{1}{\eps\sigma}$. Thus, as we will see in Section \ref{double_app_ex_sec}, for positive $\sigma$ the Dirichlet problem \eqref{double_reg} can be solved using a standard continuity argument. We then argue that for $\sigma\to 0$ we have convergence in an appropriate sense to functions $u_\eps$, which in turn for $\eps\to 0$ converge to the time of arrival function  $u:K_0\setminus K_\infty\to \mathbb{R}$, see Section \ref{pass_lim}. This solves the above difficulty (a).

More fundamentally, we use our double approximation to also solve the difficulty (b). Namely, in Section \ref{sec_estforH} and Section \ref{sec_estforAH} we prove two estimates for carefully chosen quantities at the level of the double approximators $M^{\eps,\sigma}=\textrm{graph}(u_{\eps,\sigma}/\eps)$. We choose our quantities in such a way, that on the one hand they satisfy the maximum principle and on the other hand taking the limits $\sigma\to 0$ and $\eps\to 0$ of the estimates for the double approximators yields the two main estimates for the actual level set flow. There is obviously quite some tension between these two desired properties, and we thus have to design our quantities for the double approximate estimates very carefully. For example, to estimate $\abs{A}/H$ we consider the quantity
\begin{equation}
\frac{\abs{A}+\Lambda\sigma u_{\eps,\sigma}}{(H+\sigma u_{\eps,\sigma})e^{\rho u_{\eps,\sigma}}},
\end{equation}
which turns out to indeed satisfy the maximum principle after taking in account also an improved Kato inequality at points where $\abs{A}/H$ is large, see Section \ref{sec_estforAH}. Finally, in Section \ref{pass_lim} we show that taking the limits $\sigma\to 0$ and $\eps\to 0$ of our double approximate estimates indeed yields Theorem \ref{estimate1} and Theorem \ref{estimate2}, and thus Theorem \ref{structurethm}.\\

\noindent\textbf{Acknowledgements.} We thank Brian White for bringing the problem of subsequent singularities in Riemannian manifolds to our attention. This work has been partially supported by the NSF grants DMS-1406394 and DMS-1406407. The second author wishes to thank Jeff Cheeger for his generous support during the work on this project.

\section{Existence of double approximators}\label{double_app_ex_sec}

The goal of this section is to prove the existence of double approximators.

\begin{theorem}\label{existence_thm}
If $K_0\subset N$ is a mean convex domain in a Riemannian manifold, then the Dirichlet problem \eqref{double_reg} has a unique smooth solution $u_{\eps,\sigma}$ for every $\eps,\sigma>0$.
\end{theorem}

To prove Theorem \ref{existence_thm} we will use the continuity method (see e.g. \cite{SchoenYau,Schulze} for the continuity method for related equations). Namely, we consider the Dirichlet problem
\begin{align}\label{trip_app_eq}
\textrm{div}\left(\frac{Du_{\eps,\sigma,\kappa}}{\sqrt{\eps^2+\abs{Du_{\eps,\sigma,\kappa}}^2}}\right)+\frac{\kappa}{\sqrt{\eps^2+\abs{Du_{\eps,\sigma}}^2}}&=\sigma u_{\eps,\sigma,\kappa}  &  \textrm{in} \,\,\textrm{Int}(K_0),\nonumber\\
u_{\eps,\sigma,\kappa}&= 0 & \textrm{on} \,\,\partial K_0.
\end{align}
For $\kappa=0$ the problem has the obvious solution $u_{\eps,\sigma,0}=0$. We will now derive the needed a priori estimates for $\kappa\in[0,1]$. Note first that we have the sup-bound
\begin{equation}\label{easy_sup_bound}
0 \leq u_{\de,\s,\ka}\leq \frac{\ka}{\s\de},
\end{equation}
which follows directly from the maximum principle. To proceed further, we consider the graph $M^{\de,\s,\ka}=\graph(u_{\de,\s,\ka}/\de)\subset N\times \mathbb{R}_+$. We write $\tau=\tfrac{\partial}{\partial z}$ for the unit vector in $\mathbb{R}_+$ direction, and $\nu$ for the upward pointing unit normal of $M$ (here and in the following we drop the dependence on $(\eps,\sigma,\kappa)$ in the notation when there is no risk of confusion). Written more geometrically, equation  \eqref{trip_app_eq} takes the form
\begin{equation}\label{equation_geom}
H+\sigma u=\kappa V,
\end{equation}
where $H$ is the mean curvature of $M\subset N\times \mathbb{R}_+$, and $V=\tfrac{1}{\eps}\langle \tau,\nu\rangle$.
We write $\langle\cdot,\cdot\rangle$ for the product metric on $N\times \mathbb{R}_+$, and $\nabla$ for the covariant derivative on $M$. We will frequently use the following general lemma about graphs.

\begin{lemma}\label{lap_calc}
On any graph $M\subset N\times\mathbb{R}_+$ we have
\begin{equation}\label{prod_lap}
\Delta \langle \tau,\nu \rangle=\langle \tau, \nabla H \rangle-\left(\textrm{Rc}(\nu,\nu)+|A|^2\right)\langle\tau,\nu\rangle.
\end{equation}
Moreover, the weight function $w=e^{mz}$, where $m$ is a constant, satisfies 
\begin{equation}\label{dump_der}
\nabla w= mw\tau^\top,\qquad\qquad \Delta w=\left(m^2\abs{\tau^\top}^2-m\langle \tau, \nu \rangle H\right)w,
\end{equation}
where $\tau^\top=\tau-\langle \tau,\nu\rangle \nu$ denotes the tangential part of $\tau$.  
\end{lemma}

\begin{proof}
Let $e_i$ be an orthonormal frame with $\nabla_{e_i} e_j=0$ at the point in consideration, and let $h_{ij}=A(e_i,e_j)$ be the components of the second fundamental form. Note that
\begin{equation}
\nabla \langle \tau,\nu\rangle=\nabla_{e_i}\langle \tau,\nu\rangle\, e_i=h_{ij}\langle \tau, e_j\rangle\, e_i,
\end{equation}
where here and in the following repeated indices are summed over. Using this, we compute
\begin{equation}
\Delta \langle \tau,\nu \rangle =\textrm{div}(\nabla  \langle \tau,\nu\rangle)= \nabla_{e_i}h_{ij}\langle \tau, e_j \rangle-h_{ij}h_{ij}\langle \tau, \nu \rangle.
\end{equation}
The Codazzi identity gives $\nabla_{e_i}h_{ij}=\nabla_{e_j}H+\textrm{Rc}(\nu,e_j)$. Since there is no curvature in $\tau$-direction we have $\textrm{Rc}(\nu,\tau^\top)=- \langle \tau,\nu\rangle\, \textrm{Rc}(\nu,\nu)$, and equation \eqref{prod_lap} follows.

Arguing similarly, we compute $\nabla w= \nabla_{e_i} w\, e_i=mw\langle \tau,e_i \rangle\, e_i$, and
\begin{equation}
\Delta w=\textrm{div}(\nabla w)=m^2w\langle \tau,e_i \rangle\langle \tau,e_i \rangle-mwh_{ii}\langle \tau,\nu \rangle.
\end{equation}
This proves the lemma.
\end{proof}

\begin{corollary}\label{cor_evolution_vw}
On $M^{\eps,\sigma,\kappa}=\graph(u_{\eps,\sigma,\kappa}/\eps)\subset N\times\mathbb{R}_+$ we have
\begin{multline}
\Delta (V w)=-\left(\textrm{Rc}(\nu,\nu)+|A|^2+m^2\abs{\tau^\top}^2+m (\kappa V-\sigma u)\eps V\right)Vw\\
+2m\langle \tau,\nabla(V w)\rangle+\tfrac{1}{\eps}\langle \tau, \kappa\nabla V-\sigma\nabla u \rangle w.
\end{multline}
\end{corollary}

\begin{proof}
Recall that $V=\tfrac{1}{\eps}\langle \tau,\nu\rangle$ and that $H=\kappa V-\sigma u$. Using this and the formula
\begin{equation}
\Delta (Vw)=V\Delta w+w\Delta V+\frac{2}{w}\langle \nabla w,\nabla (Vw) \rangle -\frac{2V}{w}\abs{\nabla w}^2,
\end{equation}
the claim follows from a short computation.
\end{proof}

\begin{proposition}\label{V_bd_1}
Choosing $m> 2\max_{K_0}\abs{\Rc}^{1/2}$ the function $V:M^{\eps,\sigma,\kappa}\to \mathbb{R}$ satisfies
\begin{equation}
V(x,z) \geq \min\left(\frac{1}{2\eps},\min_{\partial K_0}V\right)\cdot e^{-mz}.
\end{equation}
\end{proposition}

\begin{proof}
If $Vw$ attains its minimum on $\partial M=\partial K_0$ we are done. Suppose now $Vw$ attains its minimum at an interior point $(x_0,z_0)\in M\setminus \partial M$. If $V(x_0,z_0)\geq \tfrac{1}{2\eps}$ there is nothing to prove. Suppose now $V(x_0,z_0)< \tfrac{1}{2\eps}$. Since $M$ is the graph of $u/\eps$, we have $\langle \nabla u,\tau\rangle\geq0$, and 
since $(x_0,z_0)$ is a critical point of $Vw$, we have $\nabla V=-mV\tau^\top$, and thus $\langle \tau,\nabla V\rangle=-mV\abs{\tau^\top}^2$.  Using this, and dropping some terms with the good sign, Corollary \ref{cor_evolution_vw} implies that
\begin{equation}\label{Vw_inequatmin}
\Rc(\nu,\nu)+m^2\abs{\tau^\top}^2-m \eps\sigma u V+\tfrac{m\kappa}{\eps}\abs{\tau^\top}^2\leq 0
\end{equation}
at $(x_0,z_0)$. On the other hand, recalling that $\eps V< \tfrac12$, we have $\abs{\tau^\top}^2=1-\eps^2 V^2\geq \tfrac34$. Together with the bound $\max_{K_0}\abs{Rc}\leq \tfrac{1}{4}m^2$ and the estimate \eqref{easy_sup_bound} we thus obtain
\begin{equation}
\Rc(\nu,\nu)+m^2\abs{\tau^\top}^2-m \eps\sigma u V+\tfrac{m\kappa}{\eps}\abs{\tau^\top}^2\geq  \tfrac{1}{2}m^2+\tfrac{1}{4\eps}m\kappa>0;
\end{equation}
a contradiction. This proves the proposition.
\end{proof}

\begin{remark}\label{remark_upperlower}
Recalling that $V=(\eps^2+\abs{Du}^2)^{-1/2}$, we see that the lower bound for $V$ from Proposition \ref{V_bd_1} is equivalent to an upper bound for $\abs{D u}$.
\end{remark}

\begin{lemma}[{c.f. \cite[Thm. 7.4]{evans-spruck}}]\label{boundary_grad_est1}
There exists a constant $C=C(\de,\s,K_0)<\infty$, such that
\begin{equation}
\sup_{\partial K_0}|D u_{\de,\s,\ka}| \leq C.
\end{equation}
\end{lemma}
\begin{proof}
Let $r$ be the distance function to $\partial K_0$, and let $\delta>0$, to be chosen later,  be such that $r$ is smooth on
$T_{\delta}=\{x\in K_0\, |\, r(x)<\delta\}$.
By estimate \eqref{easy_sup_bound}, for any $C\geq\frac{1}{\s\de\delta}$, the quantity $v=C r$ satisfies $v\geq u_{\de,\s,\ka}$ on $\partial T_\delta$. We will now show that, for $C$ large enough, $v$ is a supersolution of equation \eqref{trip_app_eq}. To this end we compute
\begin{equation}
\mathrm{div}\left(\frac{D v}{\sqrt{\de^2+|D v|^2}}\right)+\ka\frac{1}{\sqrt{\de^2+|D v|^2}}-\s v \leq C \frac{\Delta r}{\sqrt{\de^2+C^2}}+\frac{1}{\sqrt{\de^2+C^2}}.
\end{equation} 
Note that $\Delta r=-H_{S_r}$ where $S_r=\{x\in \Omega \, |\, d(x,\partial \Omega)=r\}$. Since $H_0:=\min_{\partial K_0}H>0$, by the smoothness of $K_0$ and the Riccati equation, there exists a $\delta=\delta(K_0)>0$ such that $r$ is smooth on $T_\delta$ and $\Delta r \leq  -\tfrac{1}{2}H_0$ there. Thus, for $C = \max\{\tfrac{1}{\s\de\delta},\tfrac{2}{H_0}\}$, the function $v$ is a supersolution of \eqref{trip_app_eq}. Since $|D r|=1$, this implies that $\sup_{\partial K_0}|D u|\leq C$.
\end{proof}

We can now prove the main theorem of this section.

\begin{proof}[{Proof of Theorem \ref{existence_thm}}]
Note first that equation \eqref{double_reg}  is of the form 
\begin{equation}
a_{ij}(D u_{\de,\s})D_iD_ju_{\de,\s}+b(D u_{\de,\s})-\s u_{\de,\s}=0.
\end{equation}
If $u_{\de,\s}$ and $\hat{u}_{\de,\s}$ are two solutions of the Dirichlet problem, then at an interior minimum of $v=u_{\de,\s}-\hat{u}_{\de,\s}$ we have $D u_{\de,\s}=D \hat{u}_{\de,\s}$ and thus
\begin{equation}
a_{ij}(D u_{\de,\s})D_iD_j v-\s v=0,
\end{equation}
which implies $v \geq 0$. Changing the roles of $u_{\de,\s}$ and $\hat{u}_{\de,\s}$, this proves uniqueness.

To prove existence, fix $\de,\s>0$, and let 
\begin{equation}
I=\{\ka\in [0,1]\,|\, \textrm{equation \eqref{trip_app_eq} has a solution with the parameters } (\de,\s,\ka)\}.
\end{equation}
We want to show that $1\in I$. Since $0\in I$, it sufficies to show that $I$ is open and closed.

To show closeness, we first recall the sup-bound $u\leq \tfrac{1}{\eps\sigma}$ from \eqref{easy_sup_bound}, and observe  that Proposition \ref{V_bd_1}, Remark \ref{remark_upperlower} and Lemma \ref{boundary_grad_est1} give the estimate
\begin{equation}
\sup_{K_0}\abs{Du}\leq C,
\end{equation}
where $C$ is independent of $\kappa$. By DeGiorgi-Nash-Moser  and Schauder estimates we get $\ka$-independent higher derivative bounds up to the boundary for solutions of the $(\de,\s,\ka)$-problem if $\ka\in I$. If $\{\ka_m\}\subseteq I$ and $\ka_m \rightarrow \ka$,  it follows that a subsequence of $u_{\de,\s,\ka_m}$ converges to a solution $u_{\de,\s,\ka}$ of the $(\de,\s,\ka)$-problem, which implies that $\ka\in I$. \\
To show that $I$ is open, consider the operator $\mathcal{M}_\ka:C^{2,\alpha}_0(K_0)\rightarrow C^{\alpha}(K_0)$ given by
\begin{equation}
\mathcal{M}_\ka(u)=\mathrm{div}\left(\frac{D u}{\sqrt{\de^2+|D u|^2}}\right)+\frac{\ka}{\sqrt{\de^2+|D u|^2}}-\s u.
\end{equation}
Assuming $\ka\in I$,  its linearization at $u_{\de,\s,\ka}$ is given by 
\begin{equation}
\mathcal{L}_{\ka}(v)=\mathrm{div}\left(\frac{D v}{\sqrt{\de^2+|D u_{\de,\s,\ka}|^2}}-\frac{\langle D u_{\de,\s,\ka},D v \rangle D u_{\de,\s,\ka}}{\left(\de^2+|D u_{\de,\s,\ka}|^2\right)^{3/2}} \right)-\frac{\ka\langle D u_{\de,\s,\ka},D v \rangle}{\left(\de^2+|D u_{\de,\s,\ka}|^2\right)^{3/2}}-\s v.
\end{equation}
Note that at a positive maximum of $v$,
\begin{equation}
\mathcal{L}_{\ka}(v)\leq \frac{1}{\sqrt{\de^2+|D u_{\de,\s,\ka}|^2}}\left(\Delta v-\frac{\mathrm{Hess}\,v\,(D u_{\de,\s,\ka},D u_{\de,\s,\ka})}{\de^2+|D u_{\de,\s,\ka}|^2}\right)-\sigma v< 0, 
\end{equation}
and similarly at a negative minimum point, $\mathcal{L}_{\ka}(v)> 0$. Hence, $v=0$ is the unique solution to $\mathcal{L}_{\ka}(v)=0$ with zero boundary. Thus, by standard elliptic theory, the map $\mathcal{L}_{\ka}:C^{2,\alpha}_0(K_0)\rightarrow C^{\alpha}(K_0)$ is invertible, and by the inverse function theorem, the map $\mathcal{M}:[0,1]\times C^{2,\alpha}_0(K_0)\rightarrow [0,1]\times  C^{\alpha}(K_0)$ given by $\mathcal{M}(\ka,u)=(\ka,\mathcal{M}_\ka(u))$ is locally invertible. Taking also into account the higher derivative estimates we conclude that $I$ is open, and we are done. 
\end{proof}

\section{Double approximate estimate for $H$}\label{sec_estforH}

The goal of this section is to derive a lower bound for the mean curvature. As explained in the introduction, we will work at the level of the double approximators $M^{\eps,\s}=\graph(u_{\eps,\sigma}/\eps)$, where $u_{\eps,\sigma}$ is a solution of \eqref{double_reg} with $\eps,\sigma\in(0,1)$. The task is then to find a suitable quantity that on the one hand satisfies the maximum principle and on the other hand gives the desired mean curvature bound in the limit $\sigma,\eps\to 0$. It turns out that for the mean curvature estimate the quantity $H+\sigma u_{\eps,\sigma}$ does the job.

\begin{theorem}\label{V_bd_2}
There exist constants $c=c(K_0)>0$ and $\rho=\rho(K_0)<\infty$ such that  
\begin{equation}
H\left(x,\tfrac{1}{\eps}u_{\eps,\sigma}(x)\right)+\sigma u_{\eps,\sigma}(x) \geq c e^{-\rho u_{\de,\s}(x)}.
\end{equation}
for every $x\in K_0$, whenever $u_{\eps,\sigma}$ is a solution of \eqref{double_reg} with $\eps,\sigma\in(0,1)$.
\end{theorem}

\begin{remark}\label{lim_H_rk}
Taking the limits $\sigma\to 0$ and $\eps\to 0$ the estimate from Theorem \ref{V_bd_2} yields the mean curvature lower bound from Theorem \ref{estimate1}, see Section \ref{pass_lim} for the proof.
\end{remark}

In view of the equation $V=H+\sigma u$, proving Theorem \ref{V_bd_2} amounts to improving the lower bound for $V$ from Section \ref{double_app_ex_sec} in two ways. Namely, we will argue that in the case $\kappa=1$ the factor $e^{-mz}$ in Proposition \ref{V_bd_1} can be replaced by the better factor $e^{-\rho\eps z}$, and we will replace Lemma \ref{boundary_grad_est1} by a boundary estimate which is uniform in $\eps$ and $\sigma$.

\begin{proposition}\label{V_better_bd}
Choosing $\rho>4\max_{K_0}\abs{\Rc}$ the function $V:M^{\eps,\sigma}\to \mathbb{R}$ satisfies
\begin{equation}
V(x,z) \geq \min\left(\frac{1}{2\eps},\min_{\partial K_0}V\right)\cdot e^{-\eps\rho z}.
\end{equation}
\end{proposition}

\begin{proof}
Consider the function $Vw$ where $w=e^{\rho\eps z}$. As in the proof of Proposition \ref{V_bd_1} we can assume that $Vw$ attains its minimum at an interior point $(x_0,z_0)\in M\setminus\partial M$ and that $V(x_0,z_0)<\tfrac{1}{2\eps}$ (otherwise there is nothing to prove).
The estimate \eqref{Vw_inequatmin} with $\kappa=1$ and $m=\rho\eps$ reads
\begin{equation}
\Rc(\nu,\nu)- \eps^2\rho \sigma u V+(\rho+\eps^2\rho^2)\abs{\tau^\top}^2\leq 0.
\end{equation}
Combining this with the inequalities $\eps\sigma u\leq 1$, $V<\tfrac{1}{2\eps}$, and $\abs{\tau^\top}\geq \tfrac34$ yields
\begin{equation}
\Rc(\nu,\nu)+\tfrac14 \rho< 0,
\end{equation}
which contradicts our choice of $\rho$. This proves the proposition.
\end{proof}

\begin{lemma}[Uniform boundary estimate]\label{boundary_grad_est2} 
There exists a constant $C=C(K_0)<\infty $ such that
\begin{equation}
\sup_{\partial K_0}|D u_{\de,\s}| \leq C.
\end{equation}
\end{lemma}

\begin{proof}
As in the proof of Lemma \ref{boundary_grad_est1} we will construct a suitable barrier function, but this time by bending the smooth solution to infinity (c.f. \cite[Lemma 18]{BM}).

By mean convexity, for $T_0=T_0(K_0)>0$ small enough the restricted time of arrival function $u:K_0\setminus K_{T_0}\to \mathbb{R}$ is smooth and satisfies the estimates
\begin{equation}\label{smooth_der_bds}
C^{-1} \leq |D u| \leq C,\;\;\;\; |\mathrm{Hess} u|\leq C,
\end{equation}
for some $C=C(K_0)<\infty$. Recall also that $u$ satisfies the equation
\begin{equation}\label{levelseteq}
\mathrm{div}\left(\frac{Du}{\abs{Du}}\right)+\frac{1}{\abs{Du}}=0.
\end{equation}
For $T\in(0,T_0)$, let $\p:[0,T)\rightarrow [0,\infty)$, $\p(t)=\frac{1}{T-t}-\frac{1}{T}$. We will now show that for $T$ small enough the function $v=\p(u)$ is a supersolution of equation \eqref{double_reg}. To this end, we compute
\begin{align}
&\mathrm{div}\left(\frac{D v}{\sqrt{\de^2+|D v|^2}}\right)=\mathrm{div}\left(\frac{\p' Du }{\sqrt{\de^2+|\p' D u|^2}}\right)\\
&\qquad=\frac{\p''|D u|^2}{\sqrt{\de^2+|D v|^2}}+\frac{\p'}{\sqrt{\de^2+|D v|^2}}\left(\Delta u -\frac{\p'\p''|D u|^4+\p'^2\mathrm{Hess}\,u\, (D u,D u)}{\de^2+|D v|^2} \right)\nonumber\\
&\qquad=\frac{\de^2\p''|D u|^2}{\left(\de^2+|D v|^2\right)^{3/2}}-\frac{\p'}{\sqrt{\de^2+|D v|^2}}+\frac{\eps^2\p'}{(\eps^2+\abs{Dv}^2)^{3/2}} \mathrm{Hess}\,u\left(\frac{Du}{\abs{Du}},\frac{Du}{\abs{Du}}\right),\nonumber
\end{align}
where we used equation \eqref{levelseteq} in the last step. Now observe that
\begin{equation}
\frac{|D u|^2}{\de^2+|D v|^2} \leq \frac{1}{\p'^2},\qquad\qquad \frac{\p'}{\eps^2+\abs{Dv}^2}\leq \frac{1}{\p' \abs{Du}^2}.
\end{equation}
Thus, taking also into account \eqref{smooth_der_bds} we conclude that
\begin{multline}
\sqrt{\de^2+|D v|^2}\left(\textrm{div}\left( \frac{Dv}{\sqrt{\eps^2+\abs{Dv}^2}}\right)-\sigma v\right)+1,\\
\leq 2\eps^2(T-t)-\frac{1}{(T-t)^2}+C\eps^2 (T-t)^2+1
\end{multline}
which is negative if $T=T(K_0)$ is sufficiently small. Thus, for such $T$, the function $v$ is a supersolution of equation \eqref{double_reg} with $v=0$ on $\partial K_0$ and $v\to\infty$ on $\partial K_T$. Therefore,
\begin{equation}
\sup_{\partial K_0}|D u_{\de,\s}| \leq \sup_{\partial K_0}|D v| \leq \frac{C}{T^2}.
\end{equation}
This proves the lemma.
\end{proof}

\begin{remark}[Uniform lower bound]\label{remark_uniformlower}
Similarly, considering the function $\p(t)=ct(T-t)$ we see that there is a constant $c=c(K_0)>0$ such that $\inf_{\partial K_0}\abs{Du_{\eps,\sigma}}\geq c>0$.
\end{remark}

\begin{proof}[{Proof of Theorem \ref{V_bd_2}}]
Recalling that $V=H+\sigma u_{\eps,\sigma}=(\eps^2+\abs{Du_{\eps,\sigma}}^2)^{-1/2}$, the theorem follows by combining Proposition \ref{V_better_bd} and Lemma \ref{boundary_grad_est2}.
\end{proof}

\section{Double approximate estimate for $|A|/H$}\label{sec_estforAH}

The purpose of this section is to prove the following estimate.
\begin{theorem}\label{main_curv_bound}
There exist constants $\rho=\rho(K_0)<\infty$, $C=C(K_0)<\infty$, $\de_0=\de_0(K_0)>0$ and $\s_0=\s_0(K_0)>0$, such that
\begin{equation}
\frac{\abs{A}\left(x,\tfrac{1}{\eps}u_{\eps,\sigma}(x)\right)}{H\left(x,\tfrac{1}{\eps}u_{\eps,\sigma}(x)\right)+\sigma u_{\eps,\sigma}(x)} \leq C e^{\rho u_{\de,\s}(x)}
\end{equation}
for all $x\in K_0$, whenever $u_{\eps,\sigma}$ is a solution of \eqref{double_reg} with $\eps<\eps_0$ and $\sigma<\sigma_0$.
\end{theorem}

\begin{remark}
Taking the limits $\sigma\to 0$ and $\eps\to 0$ the estimate from Theorem \ref{main_curv_bound} yields the estimate for $\abs{A}/H$ from Theorem \ref{estimate2}, see Section \ref{pass_lim} for the proof.
\end{remark}

We will prove Theorem \ref{main_curv_bound} by applying the maximum principle to the function
\begin{equation}\label{def_ofg}
G=\frac{\abs{A}+\Lambda\sigma u}{Vw},
\end{equation}
where $V=H+\sigma u$, $w=e^{\eps\rho z}$, and where $\rho<\infty$ and $\Lambda<\infty$ will be specified later. As will become clear below, the extra term $\Lambda \sigma u$ is crucial for the maximum principle. We begin by computing the Laplacian of the norm of the second fundamental form.
 
\begin{proposition}\label{A_lap}
At any interior point with $\abs{A}\neq 0$ we have
\begin{equation}
\Delta |A| -\tfrac{|\nabla A|^2-|\nabla|A||^2}{|A|}\geq \tfrac{1}{\de}\langle \tau,\nabla |A| \rangle-\abs{A}^3 -C\s u|A|^2-C\max\left(1,\sigma u,\abs{A}\right).
\end{equation} 
\end{proposition}
\begin{proof}
We recall Simon's inequality for hypersurfaces in Riemannian manifolds \cite{Simons_identity},
\begin{equation}
\tfrac{1}{2}\Delta |A|^2-|\nabla A|^2 \geq \langle A,\nabla^2 H \rangle -|A|^4+H\textrm{tr}(A^3)-C(|A|+|A|^2),
\end{equation}
where $C=C(\max_{K_0}\abs{\textrm{Rm}},\max_{K_0}\abs{\nabla\textrm{Rm}})$. To find the Hessian of the mean curvature in our case we use the formula $H=\tfrac{1}{\eps}\langle \tau,\nu\rangle-\sigma u$, and compute (c.f. Lemma \ref{lap_calc}):
\begin{equation}
\nabla^2 \langle \tau,\nu\rangle=\nabla_{\tau^\top} A-\left(A^2+\textrm{Rm}(\nu,\cdot,\nu,\cdot)\right)\langle \tau,\nu\rangle,
\end{equation}
and
\begin{equation}\label{hessianofu}
\nabla^2 u=-\eps \langle \tau,\nu\rangle A.
\end{equation}
It follows that
\begin{equation}
\langle A,\nabla^2 H\rangle \geq \tfrac{1}{\eps} \langle A,\nabla_{\tau^\top} A\rangle -(H+\sigma u)\textrm{tr}(A^3)-C(\sigma u+\abs{A})\abs{A},
\end{equation}
and thus
\begin{equation}
\tfrac{1}{2}\Delta |A|^2 -\abs{\nabla A}^2\geq \frac{1}{2\de}\langle \tau,\nabla |A|^2 \rangle-\abs{A}^4-\sigma u\textrm{tr}(A^3)-C(1+\sigma u+\abs{A})\abs{A}.
\end{equation}
This implies the claim.
\end{proof}

To make use of the gradient term, we prove the following improved Kato inequality.

\begin{proposition}\label{strong_kato}
There exist constants $c=c(n)<1$ and $C=C(\max_{K_0}\abs{\textrm{Rm}})<\infty$ such that
\begin{equation}
\abs{\D{\abs{A}}}\leq c \abs{\D A}+2\abs{\D H}+C.
\end{equation}
\end{proposition}
\begin{proof}
 For any unit vector $X$, we will derive an estimate for the quantity
\begin{equation}
\abs{A}\abs{\D_X\abs{A}}=\frac{1}{2}\abs{\D_X\abs{A}^2}=\abs{\langle \D A,X\otimes A\rangle}\, .
\end{equation}
Let $(\D A)^{\textrm{sym}}$ be the totally symmetric part of the 3-tensor $\D A$, i.e.
\begin{equation}
(\D A)^{\textrm{sym}}_{ijk}=\tfrac{1}{3}(\D_i A_{jk}+\D_j A_{ik}+\D_k A_{ij}).
\end{equation}
Using the Codazzi identity and the bound $\abs{\textrm{Rm}}\leq C$ we see that
\begin{equation}\label{est_cod1}
\abs{(\D A)^{\textrm{sym}}-\D A}\leq C.
\end{equation}
Next, observe that any totally symmetric 3-tensor $T$ can be decomposed as $T=T^{\textrm{tr}}+T^0$, where
\begin{equation}
T^{\textrm{tr}}_{ijk}=\frac{1}{n+2}\left( T_{ppi}g_{jk}+T_{ppj}g_{ik}+T_{ppk}g_{ij} \right)
\end{equation}
is the trace-part, and $T^0$ is the totally traceless part.\\
Using again the Codazzi identity and the bound $\abs{\textrm{Rm}}\leq C$ we see that
\begin{equation}\label{est_cod2}
\abs{(\D A)^{\textrm{sym,tr}}}\leq \frac{3\sqrt{n}}{n+2}\abs{\nabla H}+C.
\end{equation}
Combining \eqref{est_cod1} and \eqref{est_cod2} we obtain the estimate
\begin{equation}
\abs{\langle \D A,X\otimes A\rangle}\leq \abs{\langle (\D A)^{\textrm{sym,0}},X\otimes A\rangle}+\frac{3\sqrt{n}}{n+2}\abs{\nabla H}\abs{A}+C\abs{A}.
\end{equation}
Observing that $\abs{(\D A)^{\textrm{sym,0}},(X\otimes A)\rangle}\leq \abs{\D A} \abs{(X\otimes A)^{\textrm{sym,0}}}$, the remaining task is to estimate the norm of $(X\otimes A)^{\textrm{sym,0}}$. This can be done by a straightforward computation:
\begin{align}
\abs{(X\otimes A)^{\textrm{sym,0}}}^2&=\abs{(X\otimes A)^{\textrm{sym}}}^2-\frac{4}{3(n+2)}\abs{A(X,.)}^2\\
&=\frac{1}{3}\abs{A}^2+\left(\frac{2}{3}-\frac{4}{3(n+2)}\right)\abs{A(X,.)}^2\, .
\end{align}
Putting everything together, the proposition follows.
\end{proof}

We will apply Proposition \ref{strong_kato} in combination with the following lemma.

\begin{lemma}\label{lemma_gradh}
At any critical point of $G$ we have the estimate
\begin{equation}
\abs{\nabla H}\leq \frac{V}{\abs{A}}\abs{\nabla{\abs{A}}}+\frac{1}{\abs{A}}\eps\sigma \Lambda V+\eps\rho V+\eps\sigma.
\end{equation}
\end{lemma}

\begin{proof}
The equation $\nabla \log G=0$ can be written in the form
\begin{equation}\label{eq_critical}
\frac{1}{V}(\nabla H+\sigma \nabla u)=\nabla\log(\abs{A}+\Lambda\sigma u)-\nabla\log w.
\end{equation}
Observing that $\nabla\log w=\eps\rho\tau^\top$ and $\nabla u=\eps \tau^\top$, and solving for $\nabla H$ we obtain
\begin{equation}
\nabla{H}=\left(\frac{\nabla{\abs{A}}+\Lambda\sigma\eps \tau^\top}{\abs{A}+\Lambda\sigma u}-\eps\rho\tau^\top\right)V-\eps\sigma \tau^\top.
\end{equation}
The claim follows.
\end{proof}

We are now ready to prove the main theorem of this section.

\begin{proof}[Proof of Theorem \ref{main_curv_bound}]
Throughout the proof we write $C=C(K_0)<\infty$ for a constant that can change from line to line. This should not be confused with $c=c(n)<1$, which is a fixed dimensional constant given by Proposition \ref{strong_kato}. 

Consider the function $G$ defined in \eqref{def_ofg}. The parameters $\rho$ and $\Lambda$ will be specified in the last line of the proof (depending only on the dimension and geometry of $K_0$). For now, we only impose the condition that $\rho\geq 2\rho_1$, where $\rho_1=\rho_1(K_0)$ is the constant from Theorem \ref{V_bd_2}. We will choose $\eps_0=\sigma_0=\max(\rho,\Lambda)^{-1}$. Thus, tacitly assuming that $\eps<\eps_0$ and $\sigma<\sigma_0$, we have inequalities like $\sigma \Lambda< 1$ and $\eps \rho< 1$ at our disposal.

Theorem \ref{V_bd_2}, Remark \ref{remark_uniformlower} and DeGiorgi-Nash-Moser and Schauder estimates up to the boundary give a uniform upper bound for $\sup_{\partial K_0}G$. Thus, if the maximum of $G$ occurs at the boundary $\partial  K_0$ we are done.
 Suppose now the maximum of $G$ is attained at an interior point $(x_0,z_0)\in M\setminus \partial M$. If $\abs{A}< \tfrac{4}{1-c} \max(1,\sigma u,V)$ at $(x_0,z_0)$, then Theorem \ref{V_bd_2} together with the constraint $\rho\geq 2\rho_1$ yields $G\leq C$ and we are done. Suppose now
\begin{equation}\label{ass_max}
\abs{A}\geq \tfrac{4}{1-c} \max(1,\sigma u,V).
\end{equation}
Condition \eqref{ass_max} will allow us to absorb lower order terms. For example, all lower order terms in the inequality from Lemma \ref{lemma_gradh} can be safely estimated by $C\abs{A}$, giving:
\begin{equation}
\abs{\nabla{H}}\leq \tfrac{1-c}{4}\abs{\nabla{\abs{A}}}+C\abs{A}.
\end{equation}
Combining this with Proposition \ref{strong_kato} and using again condition \eqref{ass_max} we infer that
\begin{equation}\label{improved_kato_version}
\abs{\nabla A}^2-\abs{\nabla\abs{A}}^2\geq \delta \abs{\nabla\abs{A}}^2-C\abs{A}^2,
\end{equation}
for some $\delta=\delta(n)>0$. This will be an important ingredient for the estimate below.

Since $(x_0,z_0)$ is a maximum point of $G$ we have $\Lap G\leq 0$ and thus
\begin{equation}\label{longeq1}
\Delta(|A|+\Lambda\s u)Vw-(|A|+\Lambda\s u)\Delta(Vw) \leq 0,
\end{equation}
where we also used that $\nabla G=0$. Using Proposition \ref{A_lap}, the improved Kato estimate \eqref{improved_kato_version}, the trace of equation \eqref{hessianofu} and condition \eqref{ass_max} we obtain
\begin{equation}\label{longeq2}
\Delta(|A|+\Lambda\s u)\geq \tfrac{\delta \abs{\nabla \abs{A}}^2}{|A|}+ \tfrac{1}{\de}\langle \tau,\nabla |A| \rangle-\abs{A}^3 -C\s u |A|^2-C\abs{A}.
\end{equation}
Similarly, by Corollary \ref{cor_evolution_vw} and condition \eqref{ass_max} we have
\begin{equation}\label{longeq3}
-\Delta (V w)\geq\left(-C+|A|^2+\rho\langle\tau,\nu\rangle^2-\sigma u\right)Vw-2\eps\rho\langle \tau,\nabla(V w)\rangle-\tfrac{1}{\eps}\langle \tau, \nabla V \rangle w.
\end{equation}
When substitution \eqref{longeq2} and \eqref{longeq3} into \eqref{longeq1} we will use the following claim.
\begin{claim}\label{claim_gradientterms}
The contribution from the $\langle \tau,\nabla\,\cdot\,\rangle$-terms can be estimated as:
\begin{multline}
\tfrac{1}{\eps}\langle\tau,\nabla\abs{A}\rangle Vw-\left(\abs{A}+\Lambda\sigma u\right)\left(2\eps\rho \langle \tau,\nabla (Vw)\rangle+\tfrac{1}{\eps}\langle \tau,\nabla V\rangle w\right)\\
\geq \left(-2\eps\rho\langle \tau,\nabla\abs{A}\rangle+(\rho\abs{A}-3)\abs{\tau^\top}^2 \right)Vw.
\end{multline}
\end{claim}

\begin{proof}[{Proof of Claim \ref{claim_gradientterms}}]
The equation $\nabla \log G=0$ can be written in the form
\begin{equation}
(\abs{A}+\Lambda\sigma u)\nabla(Vw)=Vw(\nabla\abs{A}+\Lambda\sigma\eps \tau^\top).
\end{equation}
Using this, and the formula $\nabla(Vw)=w\nabla V+Vw\eps\rho \tau^\top$, we compute
\begin{multline}
\tfrac{1}{\eps}\langle\tau,\nabla\abs{A}\rangle Vw-\left(\abs{A}+\Lambda\sigma u\right)\left(2\eps\rho \langle \tau,\nabla (Vw)\rangle+\tfrac{1}{\eps}\langle \tau,\nabla V\rangle w\right)\\
=\left(-2\eps\rho \langle \tau,\nabla\abs{A}\rangle +\left(\rho(\abs{A}+\Lambda\sigma u)-(1+2\eps^2\rho)\sigma \Lambda\right)\abs{\tau^\top}^2 \right)Vw.
\end{multline}
Dropping the term $\rho \Lambda\sigma u$ and estimating $(1+2\eps^2\rho)\sigma \Lambda<3$, the claim follows.
\end{proof}

Now, substituting \eqref{longeq2} and \eqref{longeq3} into \eqref{longeq1}, and using Claim \ref{claim_gradientterms}, we arrive at
\begin{multline}\label{long_ineq}
0\geq \tfrac{\delta \abs{\nabla \abs{A}}^2}{|A|}-\abs{A}^3 -C\s u |A|^2-C\abs{A}\\
+\left(\abs{A}+\Lambda\sigma u\right)\left(-C+\abs{A}^2+\rho\langle \tau,\nu\rangle^2-\sigma u\right)
-2\eps\rho\langle \tau,\nabla\abs{A}\rangle +(\rho\abs{A}-3)\abs{\tau^\top}^2.
\end{multline}
Observe that the $\abs{A}^3$-terms cancel, and that we have the estimate
\begin{equation}
-2\eps\rho\langle \tau,\nabla\abs{A}\rangle\geq - \tfrac{\delta \abs{\nabla \abs{A}}^2}{|A|}-\delta^{-1}\abs{A}.
\end{equation}
Also note that the identity $\langle \tau,\nu\rangle^2+\abs{\tau^\top}^2=1$ enables us to extract a positive term $\rho\abs{A}$. The idea is now that the good terms $\Lambda\sigma u\abs{A}^2$ and $\rho \abs{A}$ win against all other terms. Namely, from \eqref{long_ineq}, the discussion following it, and condition \eqref{ass_max} we obtain
\begin{equation}
(\tfrac{1}{2}\Lambda-C)\sigma u \abs{A}^2+(\rho -C-C\Lambda)\abs{A}\leq 0.
\end{equation}
Choosing $\Lambda=3C$ and $\rho=2(C+C\Lambda)$ this gives the desired contradiction.
\end{proof}
  
\section{Passing to the limits}\label{pass_lim}
In this final section, we explain how the double approximators $u_{\de,\s}$ converge to the arrival time $u$ of the mean curvature flow of $K_0$, and how the estimates of Theorem \ref{V_bd_2} and Theorem \ref{main_curv_bound} can be passed to the limit. This will be done in two steps, first taking the limit as $\s \rightarrow 0$ to obtain approximating translators, then taking the limit as $\de\rightarrow 0$. 

\begin{theorem}\label{app_existence}
For every $\de\in(0,\eps_0)$ there exists a relatively open set $\Omega_\de \subseteq K_0$ containing the boundary $\partial K_0$ such that the following holds.
\begin{enumerate}
\item For $\sigma\to 0$, we can take a limit $u_{\de,\s}\to u_{\eps}$ in $C^\infty_{\textrm{loc}}(\Omega_\de)$, and the limit  solves the equation
\begin{equation}\label{eq_ueps}
\begin{aligned}
\mathrm{div}\left(\frac{D u_\de}{\sqrt{\de^2+|D u_\de|^2}}\right)+\frac{1}{\sqrt{\de^2+|D u_\de|^2}}=0\qquad \textrm{in} \,\, \Omega_\eps.
\end{aligned}
\end{equation}
\item We have $u_\eps=0$ on $\partial K_0$, and $u_\eps(x)\to \infty$ uniformly as $x\to\partial\Omega_\de\setminus\partial K_0$.
\item For $(x,z)\in \graph({u_\eps}/{\eps})$ we have the estimates
\begin{align}\label{est_ueps}
H\left(x,z\right) \geq c e^{-\rho u_{\de}(x)},\quad
\frac{\abs{A}}{H}\left(x,z\right)  \leq C e^{\rho u_{\de}(x)}.
\end{align} 
\end{enumerate}
\end{theorem}

\begin{remark}
Equation \eqref{eq_ueps} says that $L^\eps_t=\{(x,z)\in \Omega_\eps \,|\, z\leq \tfrac{u_\eps(x)-t}{\eps}\}$ is a selfsimilar solution of the mean curvature flow in $N\times \mathbb{R}$, translating downwards with speed $1/\eps$.
\end{remark}

\begin{proof}[{Proof of Theorem \ref{app_existence}}]
Fix $\de\in(0,\eps_0)$. First observe that we have the monotonicity
\begin{equation}
u_{\eps,\s_1}\geq u_{\eps,\s_2}\qquad \textrm{for}\quad \s_1\leq \s_2,
\end{equation}
since $u_{\s_1}$ is a supersolution to the $(\de,\s_2)$-equation \eqref{double_reg}. Thus, for every $x\in K_0$ we can pass to a pointwise (possibly improper) limit $u_\eps(x)=\lim_{\sigma\to 0} u_{\eps,\sigma}(x)\in[0,\infty]$.

Let $\Omega_\de=\{x\in K_0 \,| \, u_\eps(x)<\infty\}$. By Theorem \ref{V_bd_2} we have the gradient estimate
\begin{equation}\label{gradestagain}
 \abs{D u_{\eps,\sigma}}\leq Ce^{\rho u_{\eps,\sigma}},
\end{equation}
where $C=C(K_0)<\infty$. By the gradient estimate, if $u_{\eps,\sigma}\leq \Lambda$ at some point, then $u_{\eps,\sigma}\leq 2 \Lambda$  in a neighborhood of definite size. In particular, $\Omega_\de\subset K_0$ is open and contains a neighborhood of the boundary $\partial K_0$. Moreover, combining the gradient estimate with DeGiorgi-Nash-Moser and Schauder estimates we see that the convergence $u_{\eps,\sigma}\to u_\eps$ is locally smooth in $\Omega_\eps$.
In particular, since we have smooth convergence we can easily take the limit $\sigma\to 0$ in \eqref{double_reg} to obtain \eqref{eq_ueps}, and take the limit $\sigma\to 0$ in Theorem \ref{V_bd_2} and Theorem \ref{main_curv_bound} to obtain \eqref{est_ueps}. Finally, suppose there is a sequence $x_i\in \Omega_\eps$ with $x_i\to x\in\partial \Omega_\eps\setminus\partial K_0$, but $\sup_i u_{\eps}(x_i)<\infty$. Then the gradient estimate gives an open neighborhood of $x$ where $u_\eps$ is bounded; this contradicts $x\in\partial \Omega_\eps\setminus\partial K_0$, and thus proves property 2.
\end{proof}

\begin{theorem}\label{u_hat_lim}
Let $K_0\subset N$ be a mean convex domain, and let $u:K_0\setminus K_\infty\to \mathbb{R}$ be the time of arrival function of its level set flow $\{K_t\}_{t\geq 0}$. Let $u_\eps:\Omega_\eps\to\mathbb{R}$, $\eps\in(0,\eps_0)$, be the family of functions given by Theorem \ref{app_existence}, and let $L^\eps_t=\{(x,z)\in \Omega_\eps \,|\, z\leq \tfrac{u_\eps(x)-t}{\eps}\}$.
\begin{enumerate}
\item For $\eps\to 0$, the functions $u_\eps$ converge locally uniformly to $u$, and the family of mean curvature flows $\{L^\eps_t\}$ converges to the mean curvature flow $\{K_t\times \mathbb{R}\}$ in the strong Hausdorff sense (see \cite{HK_meanconvex}) and in the sense of Brakke flows (see \cite{Ilmanen}).
\item The level set flow $\{K_t\}$ satisfies the estimates
\begin{equation}\label{mainesttoprove}
\inf_{\partial K_t^\textrm{reg}}H\geq c e^{-\rho t},\qquad \inf_{ \partial K_t^\textrm{reg}}\frac{\abs{A}}{H}\leq C e^{\rho t}.
\end{equation}
\end{enumerate}
\end{theorem}

\begin{proof}
By the first item of Theorem \ref{app_existence} and equation \eqref{gradestagain} we have the gradient estimate
\begin{equation}
\abs{D u_{\eps}}(x) \leq Ce^{\rho u_{\eps}(x)},
\end{equation}
 where $x\in\Omega_\eps$, and $C=C(K_0)<\infty$. The gradient estimate implies that for every sequence $\eps_k\to 0$ there exists a subsequence $\eps_k'\to 0$ and a relatively open set $\Omega \subseteq K_0$ containing the boundary $\partial K_0$  such that $u_{\eps_k'}\to \hat{u}$ locally uniformly in $\Omega$ and $u_{\eps_k'}\to\infty$ uniformly as $x\to\partial \Omega\setminus \partial K_0$.
  Since $\hat{u}$ arises as a limit of locally uniform Lipschitz functions, we can take the limit $\eps_k'\to 0$ in \eqref{eq_ueps} and infer that $\hat{u}$ solves the equation
\begin{equation}
\mathrm{div}\left(\frac{D \hat{u}}{|D \hat{u}|}\right)+\frac{1}{|D \hat{u}|}=0,
\end{equation}
in the viscosity sense. By the definition of viscosity solutions, the family of closed sets $\widehat{M}_t=\{x\in K_0\, | \, \hat{u}(x)= t\}_{t\geq 0}$ satisfies the avoidance principle, and thus is a set-theoretic subsolution of the mean curvature flow. Since $\{\partial K_t\}_{t\geq 0}$ is the maximal set theoretic subsolution starting at $\partial K_0$, we have the inclusion $\widehat{M}_t\subseteq \partial K_t$. Let $I=\{ t\in [0,\infty) \, | \, \widehat{M}_t=\partial K_t\}$. We will show that $I=[0,\infty)$. Clearly $0\in I$. 
Consider $\{t_n\}\subseteq I$ with $t_n \nearrow t<\infty$, and let $x\in \partial K_t$.
Choose $x_n\in \partial K_{t_n}$ with $x_n\to x$. Since $\hat{u}(x_n)= t_n$ and $(x_n,t_n)\to (x,t)$ it follows that $\hat{u}(x)= t$, and thus $x\in \widehat{M}_t$.
Consider now $T\in I$ and $x\in \partial K_t$ for $t\in(T,T+\delta)$. If $\delta$ is small enough, then by the gradient estimate $x\in \Omega$ and $\hat{u}(x)=t'$ for some $t'$ close to $T$. Thus, $x\in \widehat{M}_{t'}\subseteq \partial K_{t'}$. Since by mean convexity $\partial K_{t}\cap \partial K_{t'}=\emptyset$ for $t\neq t'$, it follows that $t=t'$, and thus $x\in \widehat{M}_t$. We have thus identified the limit with the unique mean convex level set flow, namely $\Omega=K_0\setminus K_\infty$, $\hat{u}=u$ and $\widehat{K}_t=K_t$. By uniqueness of the limit, the subsequential convergence $u_{\eps_k'}\to u$ actually entails a full limit.

Note that the time of arrival function of $\{L_t^\eps\}$ is given by $U_\eps(x,z)=u_{\eps}(x)-\eps z$. For $\eps\to 0$ it converges locally uniformly to $U(x,z)=u(x)$, which is the time of arrival function of $\{{K}_t\times \mathbb{R}\}$. In particular, $\{L_t^\eps\}$ converges to $\{K_t\times \mathbb{R}\}$ in the strong Hausdorff sense \cite[Def. 4.10]{HK_meanconvex}. By the compactness theorem for Brakke flows \cite[Thm. 7.1]{Ilmanen} and the uniqueness of the limit it also converges in the sense of Brakke flows.

Finally, having established the convergence, we can now use the local regularity theorem for the mean curvature flow \cite{brakke,White_regularity} to conclude that the limit for $\eps\to 0$ of the estimates in \eqref{est_ueps} yields the estimates in \eqref{mainesttoprove}.
\end{proof}

\bibliographystyle{alpha}
\bibliography{HaslhoferHershkovits_levelset}

\vspace{10mm}
{\sc Robert Haslhofer and Or Hershkovits, Courant Institute of Mathematical Sciences, New York University, 251 Mercer Street, New York, NY 10012, USA}\\

\emph{E-mail:} robert.haslhofer@cims.nyu.edu, or.hershkovits@cims.nyu.edu

\end{document}